\documentclass[a4paper, 11pt, final]{article}

\usepackage[notref,notcite]{showkeys}
\usepackage{graphicx,psfrag}
\usepackage{amssymb,amsthm,amsmath}
\usepackage{enumerate}
\usepackage{a4wide}
\usepackage{hyperref}

\numberwithin{equation}{section}
\allowdisplaybreaks[1]
\newtheorem{proposition}{Proposition}[section]

\newtheorem{lemma}[proposition]{Lemma}

\newtheorem{definition}[proposition]{Definition}

\newenvironment{proofof}[1]{\smallskip\noindent\emph{\textbf{Proof of #1.}}%
\hspace{1pt}}{\hspace{-5pt}{\nobreak\quad\nobreak\hfill\nobreak%
$\square$\vspace{8pt}\par}\smallskip\goodbreak}

\newcommand{\pint}[1]{\mathaccent23{#1}}
\newcommand{\Lloc}[1]{\mathbf{L^{#1}_{loc}}}
\newcommand{\C}[1]{\mathbf{C^{#1}}}

\newcommand{\Cc}[1]{\mathbf{C_c^{#1}}}

\newcommand{\modulo}[1]{{\left|#1\right|}}
\newcommand{\norma}[1]{{\left\|#1\right\|}}

\newcommand{\reali}{{\mathbb{R}}}

\newcommand{\naturali}{{\mathbb{N}}}

\newcommand{\BV}{\mathbf{BV}}
\newcommand{\Lip}{\mathbf{Lip}}
\newcommand{\spt}{\mathop{\mathrm{spt}}}

\renewcommand{\epsilon}{\varepsilon}
\renewcommand{\phi}{\varphi}
\renewcommand{\theta}{\vartheta}
\renewcommand{\L}[1]{\mathbf{L^#1}}
\newcommand{\W}[1]{\mathbf{W^{#1}}}

\renewcommand{\d}[1]{\mathinner{\mathrm{d}{#1}}}

\newcommand{\caratt}[1]{\mathbf{\chi}_{\strut #1}}

\newcommand{\Caption}[1]{
  \begin{minipage}{0.75\linewidth}
    \caption{\small{#1}}
  \end{minipage}}

 \title{A Mixed ODE-PDE Model for Vehicular Traffic}

\author{Rinaldo M.~Colombo \\ \small Unit\`a INdAM \\
  \small Universit\`a degli Studi di Brescia \\ \small Via Branze, 38
  \\ \small 25123 Brescia, Italy \\ \small \texttt{Rinaldo.Colombo@UniBs.it} \\
  \and Francesca Marcellini \\ \small Dip.~di Matematica e Applicazioni \\
  \small Universit\`a di Milano -- Bicocca \\ \small Via Cozzi, 53 \\ \small 20125 Milano, Italy \\ \small
  \texttt{Francesca.Marcellini@UniMiB.it}}
\date{}
\begin{document}

\maketitle

\begin{abstract}
  \noindent We present a traffic flow model consisting of a gluing
  between the Lighthill--Whitham and Richards macroscopic model with a
  first order microscopic follow the leader model. The basic
  analytical properties of this model are investigated. Existence and
  uniqueness are proved, as well as the basic estimates on the
  dependence of solutions from the initial data. Moreover, numerical
  integrations show some qualitative features of the model, in
  particular the transfer of information among regions where the
  different models are used.

  \medskip
  \noindent\textit{2000~Mathematics Subject Classification:} 35L65,
  90B20 \medskip

\end{abstract}

\section{Introduction}
\label{sec:Intro}

We consider a traffic flow model consisting of a macroscopic and a
microscopic descriptions glued together. The macroscopic part is
described through the Lighthill--Whitam~\cite{LighthillWhitham} and
Richards~\cite{Richards} model (LWR)
\begin{equation}
  \label{eq:LWR}
  \partial_t \rho + \partial_x \left(\rho \, v (\rho)\right) =0,
\end{equation}
which is a scalar conservation law, where the unknown $\rho = \rho
(t,x)$ is the (mean) traffic density and $v = v (\rho)$ is the (mean)
traffic speed. Microscopic models for vehicular traffic consist of a
finite set of ordinary differential equations, describing the motion
of each vehicle in the traffic flow. Below, as in~\cite{ArgallEtAl},
we consider a first order Follow--the--Leader (FtL) model, where each
driver adjusts his/her velocity to the vehicle in front, that is
\begin{equation}
  \label{eq:FtL}
  \dot p_i = v\left(\frac{\ell}{p_{i+1} - p_i}\right).
\end{equation}
Here, $p_i = p_i (t)$ is the position of the $i$-th driver, for $i=1,
\ldots, n$, and $p_{i+1} - p_i \geq \ell$ for all $i=1, \ldots, n-1$,
the fixed parameter $\ell$ denoting the (mean) vehicles' length. Here,
$\ell / (p_{i+1} - p_i)$ is the local traffic density in front of the
driver $p_i$. Equation~\eqref{eq:FtL} needs to be closed with the
trajectory of the first driver $p_n$.

In general, the two descriptions~\eqref{eq:LWR} and~\eqref{eq:FtL} can
be alternatively used in different segments of the real line. The
resulting model, in general, consists of several instances
of~\eqref{eq:LWR} and~\eqref{eq:FtL} alternated along the real line,
separated by \emph{free boundaries}, whose evolution needs to be
determined. This description enjoys the basic properties
in~\cite{GaravelloPiccoli2009} that are there considered as necessary
for a reliable description of traffic dynamics. Indeed, density and
speed are \emph{a priori} bounded, speed is never negative and
vanishes only at the maximal density.

A similar approach to traffic modeling is in~\cite{LattanzioPiccoli},
where the interface between the micro- and macro description is kept
fixed and the model in~\cite{AwRascle, Zhang2002} plays the role here
played by the LWR one. See also~\cite{ColomboMarson} for the case
$n=1$.

\par From a macroscopic point of view, vehicular traffic can be viewed
as a compressible fluid flow, whereas a microscopic approach describes
the behavior of each individual vehicle. Macroscopic descriptions
allow to simulate traffic on large networks but do not take much
account of the details. On the other hand, microscopic descriptions
can cover such details, but they are not tractable on a large
network. None of the two approaches is separately able to capture the
information of traffic dynamics. A natural strategy is therefore to
combine macroscopic and microscopic models. The result is the present
Micro--Macro Model, consisting in the coupling of the two different
descriptions.

Numerical results complete the study of the model and show the
reasonableness of it's solutions: in particular they explain how the
two micro- and macroscopic descriptions coexist in a single model,
although being separated. Below, we prove a well posedness result
separately for the LWR-FtL case, when the LWR model describes the
traffic dynamics on the right and the FtL on the left, and for the
opposite case, the FtL-LWR one; we also provide precise estimates on
how the solution depends from the initial data.

The paper is organized as follows: in the next section we introduce
the notations and the general model, when the two descriptions are
alternatively used in different segments of the real line. Then, we
prove a well posedness result separately for the LWR-FtL case and the
FtL-LWR one. In Section~\ref{sec:A} we present some numerical results
related to the model. All proofs are gathered in the last section.

\section{Notation and Main Results}
\label{sec:RP}

Throughout, we denote $\reali^+ = \left[0, +\infty\right[$ and
$\pint{\reali}^+ = \left]0, +\infty\right[$. For any $n \in \naturali$
and $\ell \in \pint{\reali}^+$, the set of admissible positions of $n$
vehicles of length $\ell$ is
\begin{equation}
  \label{eq:Pnl}
  \mathcal{P}_\ell^n
  =
  \left\{
    p \in \reali^n \colon
    p_{i+1} - p_i \geq \ell \mbox{ for } i=1, \ldots, n-1
  \right\} \,.
\end{equation}

\noindent Throughout, we assume the following condition on the speed
law:
\begin{description}
\item[(v)] $v \in \C2 ([0,1]; \reali^+)$ is strictly decreasing, with
  $v (1)=0$ and is such that $\frac{d^{2}~}{d\rho^{2}} \left(\rho \, v
    (\rho)\right) < 0$.
\end{description}

\noindent Our aim is the well posedness of a system consisting of
various instances of the LWR model~\eqref{eq:LWR} and of the FtL
model~\eqref{eq:FtL}, alternated along the real line. To this aim,
introduce the number $N \in \naturali$, $N \geq 1$, of the intervals
where the FtL model is used. Call $n_j$, with $n_j \geq 2$ for $j=1,
\ldots, N$, the number of individuals in the $j$-th interval and
denote $\mathcal{I}_p (t)$ the set of those points in $\reali$ where
the macroscopic model is used, i.e.
\begin{displaymath}
  \mathcal{I}_{p (t)}
  =
  \left]-\infty, p_1^1 (t) \right[ \cup \bigcup_{j=1}^{N-1} \left]
    p_{n_j}^j (t), p_{1}^{j+1} (t) \right[ \cup \left]p_{n_N}^N (t),
    +\infty \right[ \,,
\end{displaymath}
see Figure~\ref{fig:3}. Consider the system
\begin{equation}
  \label{eq:3}
  \!\!\!
  \left\{
    \begin{array}{@{}l@{\qquad}l@{}}
      \partial_t \rho + \partial_x \left(\rho \, v (\rho)\right) =0
      &
      \displaystyle
      x \in \mathcal{I}_{p (t)}
      \\
      \dot p_i^j (t) = v\left(\frac{\ell}{p_{i+1}^j (t) - p_i^j (t)}\right)
      & i = 1, \ldots, n_j-1\,, \qquad j=1, \ldots, N
      \\
      \dot p_{n_j}^j = v\left(\rho\left(t, p_{n_j}^j (t)\right)\right)
      & j=1, \ldots, N
      \\
      \rho (0, x) = \bar \rho (x)
      &
      \displaystyle
      x \in \mathcal{I}_{\bar p}
      \\
      p^j (0) = \bar p^j
      & j=1, \ldots, N \,,
    \end{array}
  \right.
\end{equation}
Throughout, we require that the initial data satisfy the admissibility
condition
\begin{equation}
  \label{eq:AdmInDat}
  \begin{array}{rcll}
    \bar \rho & \in & (\L1 \cap \BV) (\reali; [0,1])
    & \mbox{ with } \bar\rho (x) = x
    \mbox{ whenever } x \in \reali \setminus \mathcal{I}_{\bar p} \,,
    \\
    \bar p^j & \in & \mathcal{P}_\ell^{n_j}
    & \mbox{ for all } \quad
    j=1, \ldots, N \,.
  \end{array}
\end{equation}
\begin{figure}[!h]
  \centering
  \begin{psfrags}
    \psfrag{r}{$\rho$} \psfrag{x}{$x$} \psfrag{t}{$t$}
    \psfrag{p11}{$p_1^1$} \psfrag{p12}{$p_{n_1}^1$}
    \psfrag{p21}{$p_1^2$} \psfrag{p22}{$p_{n_2}^2$}
    \includegraphics[width=0.7\textwidth]{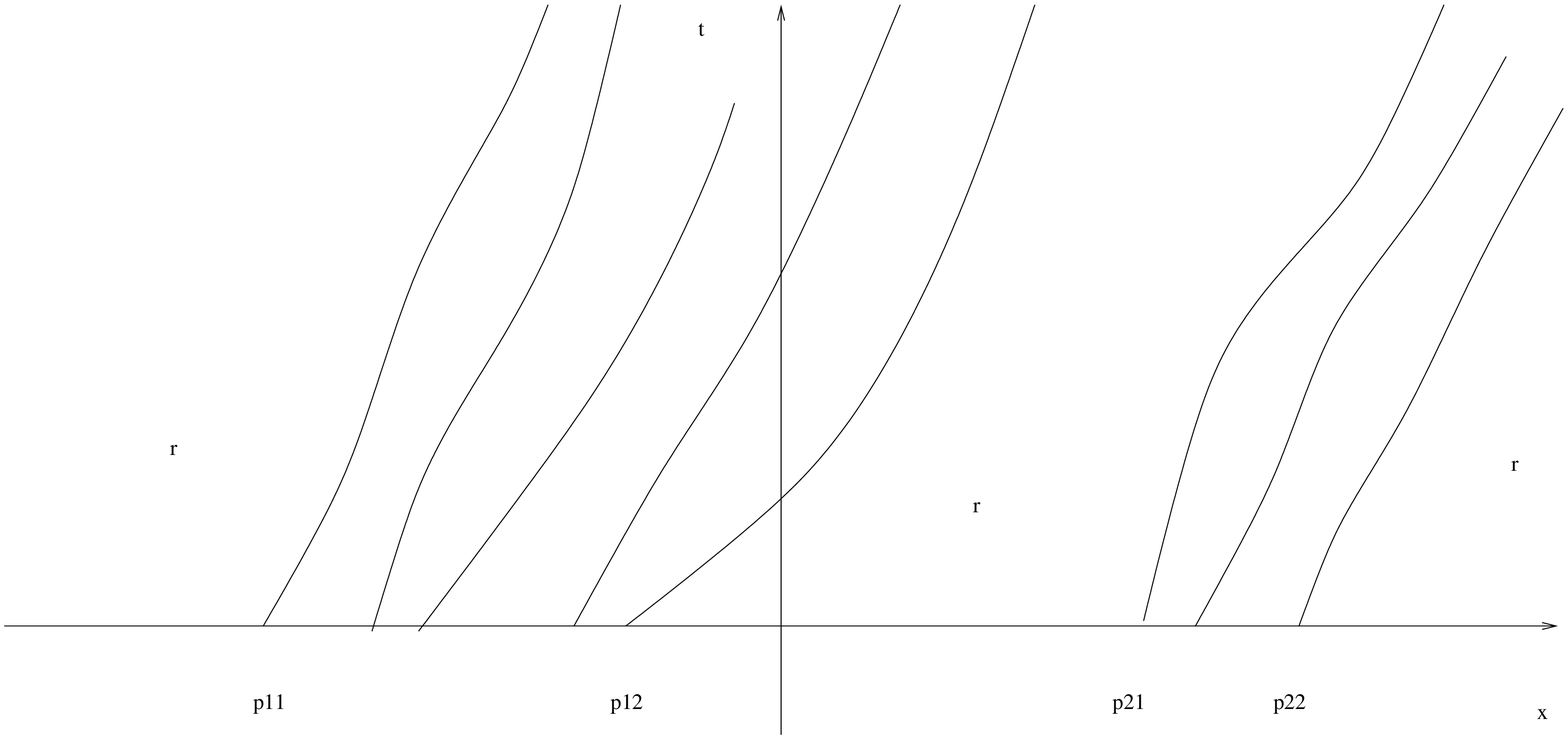}
  \end{psfrags}
  \Caption{Situation described by~\eqref{eq:3} in the case $N=2$,
    $n_1=5$ and $n_2=3$.}
  \label{fig:3}
\end{figure}

Note that problems similar to~\eqref{eq:3} can be stated equally with
the microscopic model in the rightmost and/or leftmost part of the
real line.

The first step in the rigorous treatment of~\eqref{eq:3} is the
definition of its solutions. Essentially, we require to solve the
ordinary differential equations in~\eqref{eq:3} as usual and to seek a
weak entropy (Kru\v zkov) solution to the hyperbolic conservation
law~\eqref{eq:LWR} in $\mathcal{I}_p (t)$, for $t \in \reali^+$. To
simplify the notation, we require $\rho (t, \cdot)$ to be defined on
all the real line and extend it to $0$ on $\reali \setminus
\mathcal{I}_p (t)$.

\begin{definition}
  \label{def:Sol}
  Fix positive $T$ and $\ell$, an initial distribution $\bar \rho \in
  (\L1 \cap \BV) (\reali; [0,1])$ and positions $\bar p_i^j$ for $i=1,
  \ldots, n_j$, $j=1, \ldots, N$ satisfying~\eqref{eq:AdmInDat}. A
  solution to~\eqref{eq:3} on the time interval $[0,T[$, consists of
  maps
  \begin{eqnarray*}
    \rho & \in & \C0\left([0,T]; (\L1\cap \BV) (\reali; [0,1])\right)
    \quad \mbox{ with } \rho (t,x) = 0
    \mbox{ whenever } x \in \reali \setminus \mathcal{I}_p (t)
    \\
    p^j & \in & \W{1,\infty} ([0,T]; \mathcal{P}_\ell^{n_j})
    \quad \mbox{ for } j=1, \ldots, N
  \end{eqnarray*}
  (where continuity is understood with respect to the $\L1$ topology)
  such that
  \begin{enumerate}
  \item \label{it:def:wes} for all $\phi \in \Cc1 (\left]0, T\right[
    \times \reali, \reali^+)$ with $\spt \phi \subset \left\{ (t,x)
      \in [0,T] \times \reali \colon \displaystyle x \in \mathcal{I}_p
      (t) \right\}$ the following inequality holds for all $k \in
    \reali$:
    \begin{displaymath}
      \int_0^T \int_{\reali}
      \left(
        \modulo{\rho (t,x) - k} \, \partial_t \phi (t,x)
        +
        \left(
          \rho (t,x) \, v\left(\rho (t,x)
            -
            k \, v (k)\right)
        \right)
        \partial_x \phi (t,x)
      \right)
      \d{x} \, \d{t}
      \geq 0 \,.
    \end{displaymath}
  \item \label{it:def:bd} For $j = 1, \ldots, N$ and for a.e.~$\tau
    \in \reali^+$, let $u^\tau$ be the solution to the Riemann Problem
    \begin{displaymath}
      \left\{
        \begin{array}{l}
          \partial_t u + \partial_x \left(u \, v (u)\right) =0
          \\
          u (t,x) = \left\{
            \begin{array}{l@{\quad\mbox{if }}rcl}
              \rho (\tau, p_1^j (\tau)-)
              & x & < & p_1^j (\tau) \,,
              \\
              \frac{\ell}{p_2^j (\tau) - p_1^j (\tau)}
              & x & > & p_1^j (\tau) \,.
            \end{array}
          \right.
        \end{array}
      \right.
    \end{displaymath}
    Then, $\rho (t, p_1^j (t)-) = u^\tau (t,x)$, for all $(t,x)$ such
    that $x < p_1^j (\tau) + \dot p_1^j (\tau) (t-\tau)$ and $t >
    \tau$;
  \item \label{def:it:ode} for a.e.~$t \in [0,T]$ and all $j=1,
    \ldots, N$, $i=1, \ldots, n_j-1$, $\dot p_j^i (t) =
    v\left(l\middle/\left(p_j^{i+1} (t) - p_j^i (t)\right)\right)$;

  \item \label{def:it:ode1} for a.e.~$t \in [0,T]$ and all $j=1,
    \ldots, N$, $\dot p_{n_j}^j (t) = v\left( \rho \left(t, p_{n_j}^j
        (t)+\right)\right)$.
  \end{enumerate}
\end{definition}

\noindent Above, the condition at~\ref{it:def:wes}.~is equivalent to
the usual definition of Kru\v zkov solution,
see~\cite[Formula~(6.3)]{BressanLectureNotes}. Thanks to the $\L1$
continuity in times, it also ensures the usual distributional
condition: for all $\phi \in \Cc1 (\left]-\infty, T\right] \times
\reali, \reali)$ with $\spt \phi \subset \left\{ (t,x) \in \reali^2
  \colon x \in \mathcal{I}_p (t) \mbox{ for all } t \in [0,T]
\right\}$,
\begin{displaymath}
  \int_0^T \int_{\reali}
  \left(
    \rho (t,x) \, \partial_t \phi (t,x)
    +
    \rho (t,x) \,  v\left(\rho (t,x)\right) \, \partial_x \phi (t,x)
  \right)
  \d{x} \d{t}
  +
  \int_\reali
  \bar\rho (x) \, \phi (0, x) \d{x}
  =0 \,.
\end{displaymath}

The requirement~\ref{it:def:bd}.~is the standard definition of
solution to a boundary value problem for a conservation law,
see~\cite[Definition~2.1]{DuboisLefloch},
\cite[Definition~C.1]{AmadoriColombo}
and~\cite[Definition~2.2]{ColomboGroli4}.  Remark that the
trajectories $p_1^j = p_1^j (t)$ and $p_{n_j}^j = p_{n_j}^j (t)$, for
$j = 1, \ldots, N$, are \emph{free boundaries} between micro- and
macroscopic descriptions, to be found while
solving~\eqref{eq:3}. However, only the $p_1^i$, for $i=1, \ldots, N$,
have a role in~\ref{it:def:bd}.

We remark that any solution to~\eqref{eq:3} in the sense of
Definition~\ref{def:Sol} enjoys the basic properties underlined
in~\cite{GaravelloPiccoli2009}, namely:
\begin{description}
\item[P1] Cars may have only positive speed.
\item[P2] Vehicles stop only at maximum density, i.e., the velocity
  $v$ is $0$ if and only if the density $\rho$ is equal to the maximum
  density possible.
\end{description}
\noindent The next two sections deal with the two possible gluing of
the a single instance of the LWR model and a single instance of the
FtL one.

\subsection{The Case LWR--FtL}
\label{par:LWR-FtL}

Let $n$ vehicles start at time $t = 0$ from positions $\bar p \in
\mathcal{P}_\ell^n$ and use the LWR model to describe the traffic
dynamics for $x < \bar p_1$. We are thus lead to consider the problem
\begin{equation}
  \label{eq:LWR-FtL}
  \left\{
    \begin{array}{@{}l@{\quad}rcll@{}}
      \partial_t \rho + \partial_x \left(\rho \, v (\rho)\right) =0
      & t & \in & \reali^+
      & \mbox{ and }x \, < \, p_1 (t)
      \\
      \dot p_i = v\left(\frac{\ell}{p_{i+1} - p_i}\right)
      & t & \in & \reali^+
      & \mbox{ and } i=1, \ldots, n-1
      \\
      \dot p_n = w (t)  & t & \in & \reali^+
      \\
      \rho (0,x) = \bar \rho (x) & x & \leq & \bar p_1
      \\
      p (0) = \bar p
    \end{array}
  \right.
\end{equation}
where $w \in \L\infty (\reali^+; \reali^+)$ is the speed of the
leader, $\bar \rho \in (\L1\cap\BV) (\reali; [0,1])$ describes the
vehicles' distribution for $x < \bar p_1$ and $\bar p \in
\mathcal{P}_\ell^n$. In the present case~\eqref{eq:LWR-FtL}, the
trajectory of $p_1$, i.e., $p_1 (t) + \bar p_1 + \int_0^t w (\tau)
\d\tau$, acts as a boundary between the microscopic model on its right
and the macroscopic one on its left.

Remark that from a strictly rigorous point of view,
problem~\eqref{eq:LWR-FtL} does not fit into~\eqref{eq:3}. However,
the extension of Definition~\ref{def:Sol} to the case
of~\eqref{eq:LWR-FtL} is straightforward and we omit it.

\begin{proposition}
  \label{prop:LWR-FtL}
  Fix $\ell >0$, $V > 0$, $n \in \naturali$ with $n \geq 2$ and a $v$
  that satisfies~\textbf{(v)}. Let $w$ be in $\L\infty (\reali^+; [0,
  V])$. For any $\bar p \in \mathcal{P}^n_\ell$ and for any $\bar\rho
  \in (\L1\cap \BV) (\reali; [0,1])$, problem~\eqref{eq:LWR-FtL}
  admits a unique solution in the sense of
  Definition~\ref{def:Sol}. Moreover, there exists a positive $L$ such
  that if $w' \in \L\infty (\reali^+; [0, V])$, $\bar p' \in
  \mathcal{P}^n_\ell$ and $\bar\rho' \in (\L1\cap \BV) (\reali;
  [0,1])$, then the corresponding solutions $(p, \rho)$ and $(p',
  \rho')$ satisfy for all $t\geq0$ the following estimates:
  \begin{eqnarray*}
    \norma{\rho (t,\cdot) - \rho' (t,\cdot)}_{\L1}
    & \leq &
    L \norma{\bar \rho - \bar \rho'}_{\L1}
    \\
    & + &
    L
    \left( 1+\left( 1+2V\right) \frac{2}{\ell}t\right)
    \left(
      \norma{ \bar p - \bar p'}
      +
      \norma{w-w'}_{\L1 ([0,t])}
    \right)
    \exp \left( 2\frac{\Lip(v)}{\ell}t\right)
    \\
    \norma{p (t) - p' (t)}
    & \leq &
    \left(
      \norma{\bar p - \bar p'}
      +
      \norma{w - w'}_{\L1 ([0,t])}
    \right)
    \exp \left( 2 \, \frac{\Lip(v)}{\ell} \, t \right) \,.
  \end{eqnarray*}
\end{proposition}

The proof is postponed to Section~\eqref{sec:TD}.

\subsection{The Case FtL--LWR}
\label{par:FtL-LWR}

Next we use the FtL model to describe $n$ vehicles starting at time $t
= 0$ from positions $\bar p \in\mathcal{P}_\ell^n$ and the LWR model
for $x > p_n (t)$. The free boundary between the two models is the
trajectory $p_n = p_n (t)$, chosen so that $\dot p_{n} = v \left(\rho
  (t, p_n (t))\right)$. We are thus lead to consider the problem

\begin{equation}
  \label{eq:FtL-LWR}
  \left\{
    \begin{array}{@{}l@{\quad}rcll@{}}
      \partial_t \rho + \partial_x \left(\rho \, v (\rho)\right) =0
      & t & \in & \reali^+
      & \mbox{ and }x \, > \, p_n (t)
      \\
      \dot p_i = v\left(\frac{\ell}{p_{i+1} - p_i}\right)
      & t & \in & \reali^+
      & \mbox{ and } i=1, \ldots, n-1
      \\
      \dot p_n = v \left(\rho(t, p_n (t))\right)
      & t & \in & \reali^+
      \\
      \rho (0,x) = \bar \rho (x) & x & \geq & \bar p_n
      \\
      p (0) = \bar p
    \end{array}
  \right.
\end{equation}
where $\bar \rho \in (\L1\cap\BV) (\reali; [0,1])$ describes the
macroscopic vehicles' distribution for $x > \bar p_n$ and $\bar p \in
\mathcal{P}_\ell^n$ gives the initial positions of the discrete
vehicles. In the present case~\eqref{eq:FtL-LWR}, the trajectory of
$p_n$ acts as a boundary between the microscopic model on its left and
the macroscopic one on its right.  As in the preceding section, from a
strictly rigorous point of view, problem~\eqref{eq:FtL-LWR} does not
fit into~\eqref{eq:3} but the extension of Definition~\ref{def:Sol}
to~\eqref{eq:FtL-LWR} is straightforward.

\begin{proposition}
  \label{prop:FtL-LWR}
  Fix $\ell >0$, $V > 0$, $n \in \naturali$ with $n \geq 2$ and a $v$
  that satisfies~\textbf{(v)}. For any $\bar p \in \mathcal{P}^n_\ell$
  and for any $\bar\rho \in (\L1\cap \BV) (\reali; [0,1])$,
  problem~\eqref{eq:FtL-LWR} admits a unique solution in the sense of
  Definition~\ref{def:Sol}. Moreover, there exists a positive $L$ such
  that if $v'$ satisfies~\textbf{(v)}, $\bar p' \in
  \mathcal{P}^n_\ell$ and $\bar\rho' \in (\L1\cap \BV) (\reali;
  [0,1])$, then
  \begin{equation}
    \label{eq:prima}
    \norma{\rho(t,\cdot) - \rho'(t,\cdot)}_{\L1}
    \leq
    \norma{\bar \rho - \bar \rho'}_{\L1}
    +
    \norma{\bar p - \bar p'}
  \end{equation}
  Moreover, if $\bar\rho = \bar\rho'$, there exists a non decreasing
  function $C \colon \reali^+ \to \reali^+$ such that
  \begin{equation}
    \label{eq:seconda}
    \modulo{p_{n} (t) - p'_{n} (t)}
    \leq
    \left(
      \norma{\bar p - \bar p'} + C (t) \, \norma{\bar p - \bar p'}^\alpha
    \right)
    \, \exp\left(2 \frac{\Lip (x)}{\ell} \, t\right)
  \end{equation}
  where $\alpha = \left(1 + \max _{[0,R]}\frac{v (\rho) - v (0)}{\rho
      \, v' (\rho)}\right)^{-1}$.
\end{proposition}

The proof is postponed to Section~\ref{sec:TD}.

\subsection{The General Case}
\label{par:general}

Applying iteratively Proposition~\ref{prop:LWR-FtL} and
Proposition~\ref{prop:FtL-LWR}, one obtains a general result for the
model in~\eqref{eq:3}, thanks to the finite propagation speed
in~\eqref{eq:3}.

Clearly, in the general model~\eqref{eq:3}, the number $n_j$ of
drivers in the interval $[p_i^j (t), p_{i+1}^j (t)]$ is fixed \emph{a
  priori}. An analogous property is enjoyed by the macroscopic
density, as proved by the following result.

\begin{proposition}
  \label{prop:F}
  Fix $N \in \naturali$; $n_1, \ldots, n_N$ with $n_j \geq 2$ for all
  $j$ and the initial data $\bar \rho$ and $\bar p$
  satisfying~\eqref{eq:AdmInDat}, the solution $(p,\rho)$
  to~\eqref{eq:3} satisfies:
  \begin{displaymath}
    \int_{p_{n_j}^j (t)}^{p_1^{n_{j+1}} (t)} \rho (t, x) \d{x}
    =
    \int_{\bar p_{n_j}^j}^{\bar p_1^{j+1}} \bar \rho (x) \d{x}
  \end{displaymath}
  for all $t \in \reali^+$ and for all $j=1, \ldots, N-1$.
\end{proposition}

\noindent In other words, the total amount of vehicles in each segment
$[p_{n_j}^j (t), p_1^{j+1} (t)]$ is constant.

The proof is postponed to Section~\eqref{sec:TD}.

\section{Numerical Integrations}
\label{sec:A}

To numerically integrate the models~(\ref{eq:LWR-FtL})
and~(\ref{eq:FtL-LWR}) we use the Lax-Friedrichs algorithm,
see~\cite[Section~12.1]{LeVeque}, for the partial differential
equation and the explicit forward Euler method for the ordinary
differential equation.

In the case~(\ref{eq:LWR-FtL}), we choose
\begin{equation}
  \label{eq:1}
  v (\rho) = 1-\rho \,,
  \qquad
  \ell = 0.49
  \qquad \mbox{ and } \qquad
  w (t) = 0.75
\end{equation}
with initial datum
\begin{equation}
  \label{eq:LWR-Ftl:id}
  \begin{array}{rcl}
    \bar \rho(x)
    & = &
    \caratt{[-2,-0.5]}(x)+0.8\caratt{[-7,-5]}(x)+0.6\caratt{[-9,-7]}(x)
    \\
    \bar p
    & = &
    \left[0,\, 2,\, 4,\, 6.5,\, 7,\, 7.5,\, 8,\, 8.5,\, 9,\, 9.5 \right]
    \,.
  \end{array}
\end{equation}
Note that the above choices are consistent with the assumptions
required in Proposition~\ref{prop:LWR-FtL}.
\begin{figure}[!h]
  \centering
  \includegraphics[width=0.7\textwidth]{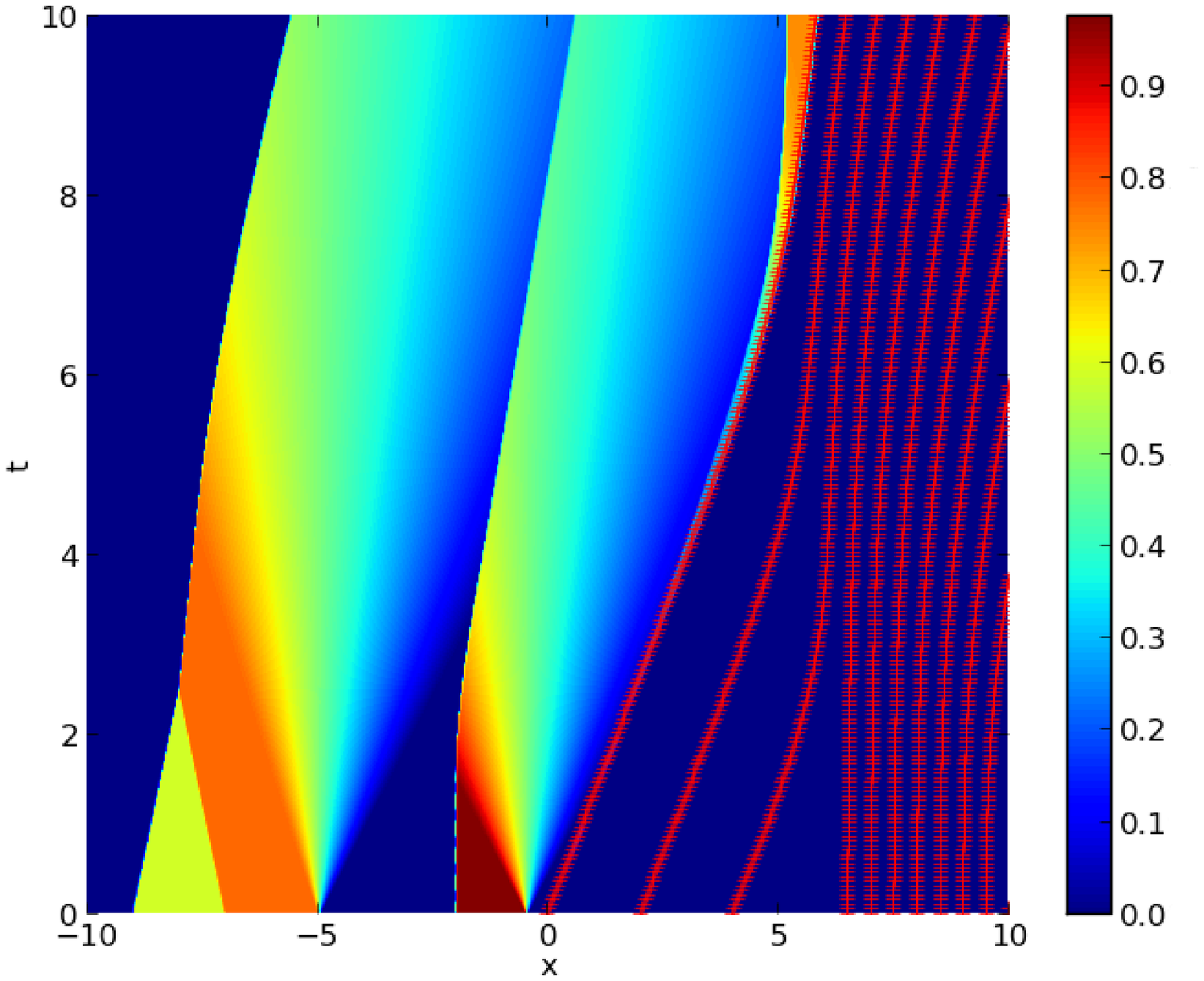}
  \Caption{Numerical integration of the LWR--FtL
    model~(\ref{eq:LWR-FtL})--(\ref{eq:1})--(\ref{eq:LWR-Ftl:id}). The
    interplay between the micro- and macroscopic phases is shown by
    the shock arising at about $t=4$, fully visible from about $t=8$.
    \label{fig:LWR-FtL}}
\end{figure}
The resulting solution is displayed in the $(t,x)$ plane in
Figure~\ref{fig:LWR-FtL}. It was computed with a space mesh size
$\Delta x = 2.5\times 10^{-3}$ and a time mesh size updated at each
time step so that
\begin{equation}
  \label{eq:mesh}
  \Delta t = 0.9 \cdot \Delta x / \Lambda
\end{equation}
$\Lambda$ being the maximal characteristic speed.

On the left, we see the typical behavior of the solutions to the LWR
model, consisting of shocks and rarefaction waves. On the right, the
microscopic part yields the trajectories of the single vehicles. Due
to the choice~(\ref{eq:LWR-Ftl:id}) of the initial datum, the cars in
front start very slowly, while the ones in the back have a higher
initial speed.  After a while these latter vehicles have to brake,
according to~(\ref{eq:FtL}). This causes the formation of a shock in
the macroscopic phase. Indeed, at about $t=4$, behind the leftmost
driver, a shock starts forming and becomes visible at about $t=8$.

\begin{figure}[!h]
  \centering
  \includegraphics[width=0.7\textwidth]{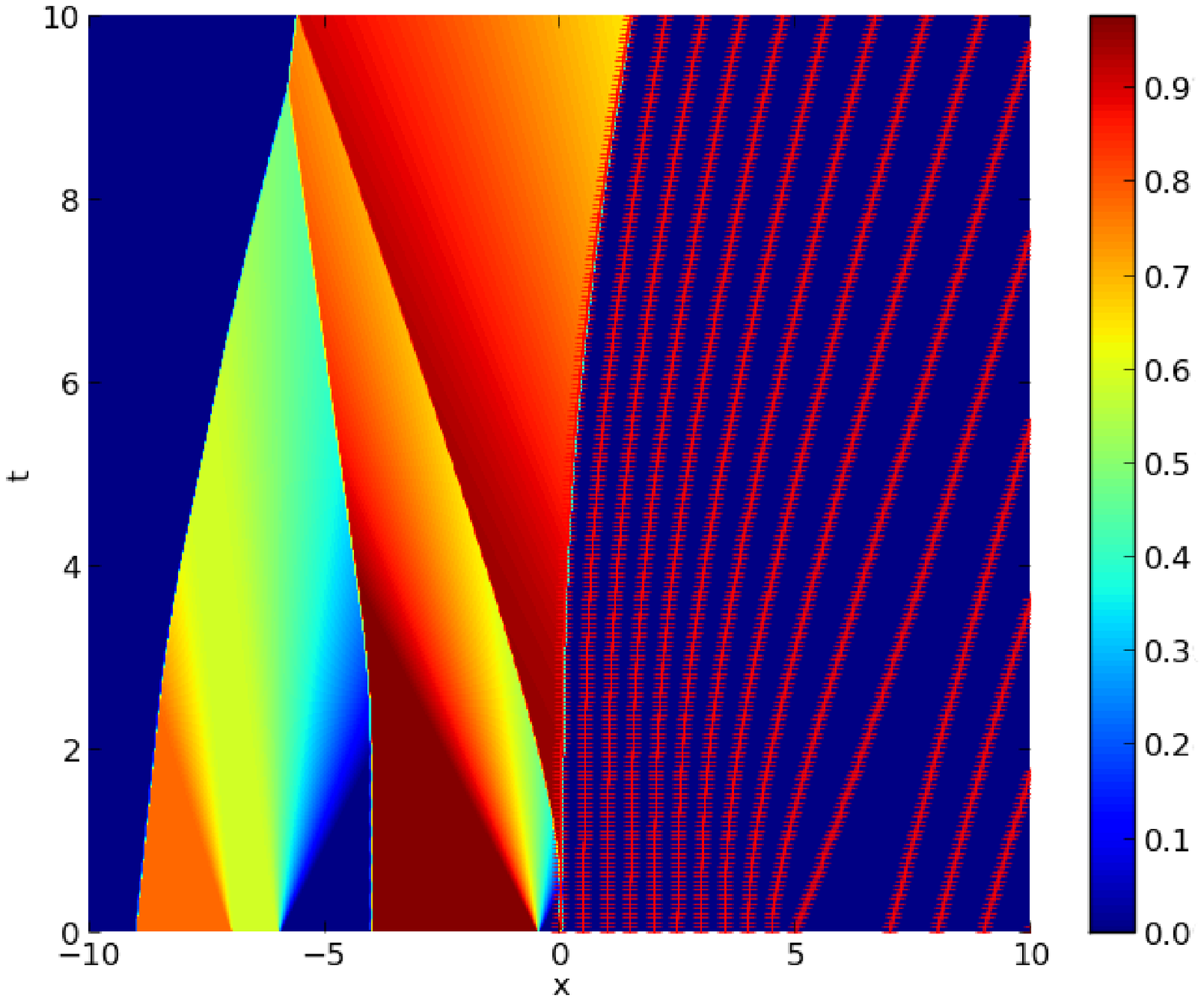}
  \Caption{Numerical integration of the LWR--FtL
    model~(\ref{eq:LWR-FtL})--(\ref{eq:1})--(\ref{eq:LWR-Ftl:id:2}). Here,
    we used the same space and time meshes as in the integration
    leading to Figure~\ref{fig:LWR-FtL}.
    \label{fig:LWR-FtL:2}}
\end{figure}
The same setting~(\ref{eq:LWR-FtL})--(\ref{eq:1}), but with initial
datum
\begin{equation}
  \label{eq:LWR-Ftl:id:2}
  \begin{array}{rcl}
    \bar \rho(x)
    & = &
    \caratt{[-4,-0.5]}(x)+0.6\caratt{[-6, -7]}(x)+0.8\caratt{[-9,-7]}(x)
    \\
    \bar p
    & = &
    \left[0,\, 0.5,\, 1.,\, 1.5,\, 2.,\,,\, 2.5,\, 3.,\, 3.5,\, 4.,\, 4.5,\, 5.,\, 7.,\, 8.,\, 9.,\, 10.\right]
    \,
  \end{array}
\end{equation}
leads to the picture in Figure~\ref{fig:LWR-FtL:2}. Here, the leftmost
drivers in the microscopic phase have a very low initial speed. Hence,
the rightmost vehicles in the macroscopic phase have to brake at about
$t=0.5$, forming a queue.
Later, the drivers in the microscopic phase accelerate and this
increase in the speeds reaches also the macroscopic phase.

\medskip

\begin{figure}[!h]
  \centering
  \includegraphics[width=0.7\textwidth]{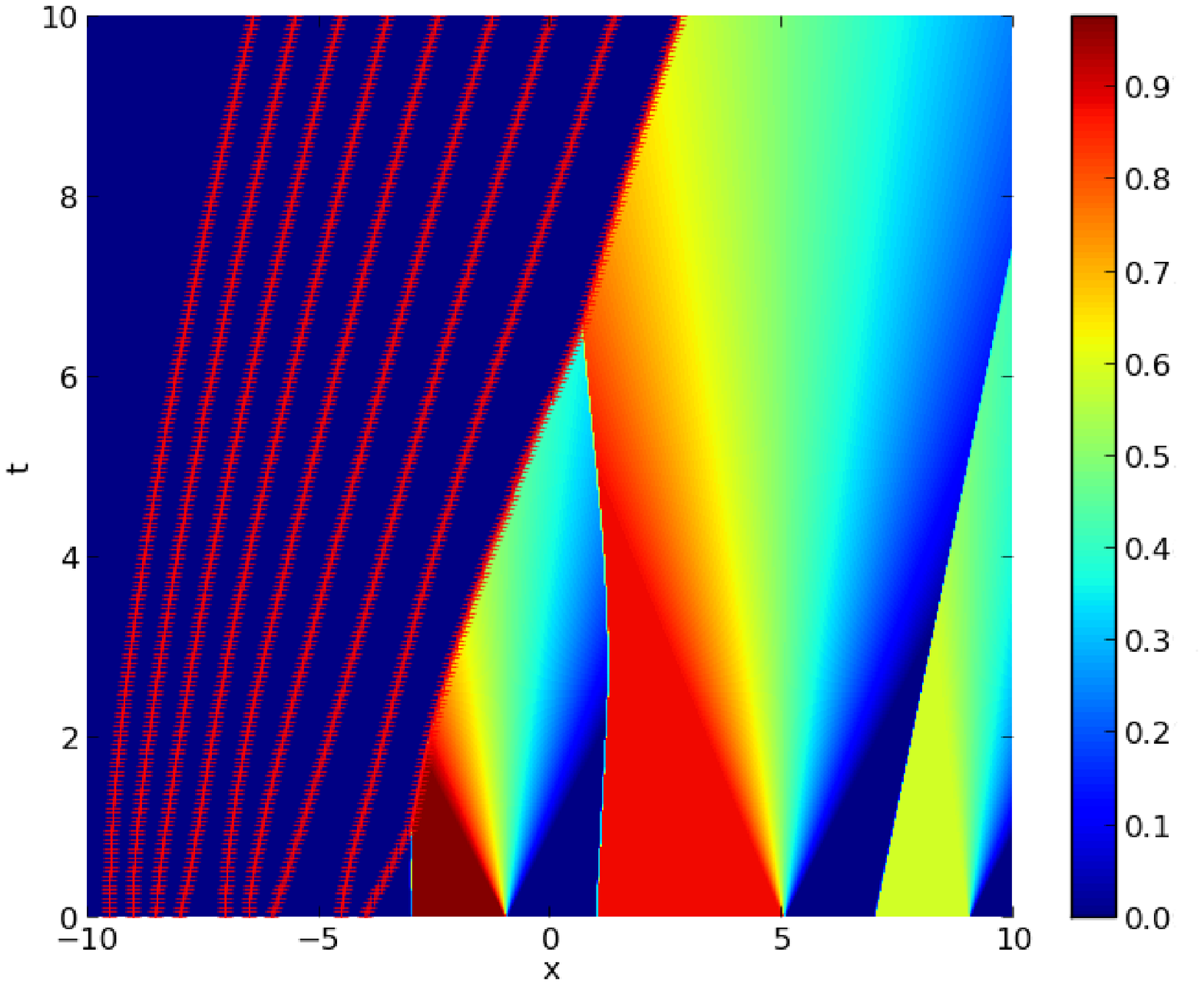}
  \Caption{Numerical integration of the FtL--LWR
    model~(\ref{eq:FtL-LWR})--(\ref{eq:1})--(\ref{eq:Ftl-LWR:id1}). The
    LWR density in the interval $[-3, \, -1]$ is maximal, hence the
    traffic speed vanishes there. As a consequence, the first vehicle
    in the microscopic phase reaches the phase boundary at about $t=2$
    and at that time its velocity is discontinuous.
    \label{fig:FtL-LWR}}
\end{figure}

 \begin{figure}[!h]
   \centering
   \includegraphics[width=0.7\textwidth]{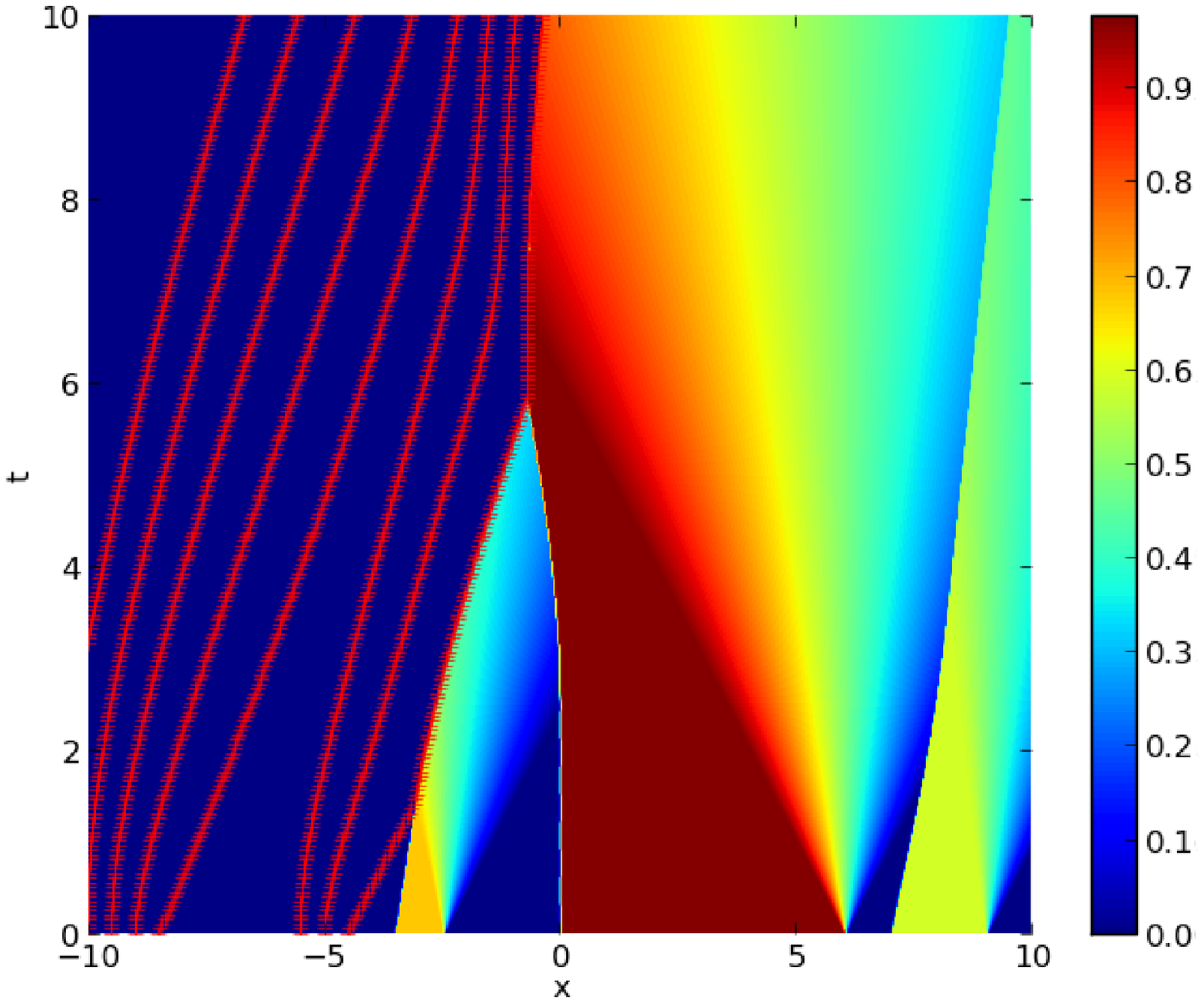}
   \Caption{Numerical integration of the FtL--LWR
     model~(\ref{eq:FtL-LWR})--(\ref{eq:1})--(\ref{eq:Ftl-LWR:id2}). Here,
     we used the same space and time meshes as in the integration
     leading to Figure~\ref{fig:FtL-LWR}. The first vehicle in the
     macroscopic phase reaches the phase boundary at about $t=1.5$ and
     at that time its velocity is discontinuous. In the macroscopic
     phase, at that time, there is an interaction between a shock and
     a rarefaction curve.
     \label{fig:FtL-LWR-2}}
 \end{figure}

 In the other case of the FtL-LWR model~(\ref{eq:FtL-LWR}), we keep
 using the choices~(\ref{eq:1}), but with the initial datum
 \begin{equation}
   \label{eq:Ftl-LWR:id1}
   \begin{array}{rcl}
     \bar \rho(x)
     & = &
     \caratt{[-3,-1]}(x)+0.9\caratt{[1,5]}(x)+0.6\caratt{[7,9]}(x)
     \\
     \bar p
     & = &
     \left[-9.5, -9, -8.5, -8, -7, -6.5, -6, -4.5, -4\right]
   \end{array}
 \end{equation}
 with a mesh $\Delta x = 10^{-3}$ and a time mesh chosen as
 in~(\ref{eq:mesh}). The resulting solution is displayed in the
 $(t,x)$ plane in Figure~\ref{fig:FtL-LWR}. Differently from what
 usually happens in the usual FtL model, here the speed of the first
 vehicle suffers a discontinuity, clearly visible at about $t=1$, due
 to its reaching the interface with the LWR phase.

 \medskip

 The same setting in~(\ref{eq:FtL-LWR}), with the
 choices~(\ref{eq:1}), but with the initial datum
 \begin{equation}
   \label{eq:Ftl-LWR:id2}
   \begin{array}{rcl}
     \bar \rho(x)
     & = &
     0.7\caratt{[-3.5,-2.5]}(x)+\caratt{[0,6]}(x)+0.6\caratt{[7,9]}(x)
     \\
     \bar p
     & = &
     \left[-11, -10, -9.5, -9, -8.5, -5.5, -5, -4.5\right]
   \end{array}
 \end{equation}
 leads to the representation in Figure~\ref{fig:FtL-LWR-2}.

 The initial density in the LWR phase is maximal in the interval
 $\left[0,6 \right]$. This situation has consequences also the
 microscopic phase. First, the speed of the leader suffers a
 discontinuity, clearly visible at about $t=1.5$, due to its reaching
 the interface with the LWR phase. Then, the drivers behind the leader
 have to brake.

 \medskip

 The figures above explain how the two micro- and macroscopic
 descriptions coexist in a single model. There is a clear backward
 propagating exchange of information between the different phases,
 although there is no exchange of mass.

 \section{Technical Details}
 \label{sec:TD}

 The following Lemma deals with the ordinary differential
 system~\eqref{eq:FtL}. Its proof reminds that
 of~\cite[Proposition~4.1]{ColomboMarcelliniRascle}.

 \begin{lemma}
   \label{lem:ode}
   Let $v$ satisfy~\textbf{(v)} and $\ell >0$. Choose $\bar p \in
   \mathcal{P}_\ell^n$. Let $w \in \Lloc1 (\reali^+; \reali^+)$. Then,
   the Cauchy problem
   \begin{equation}
     \label{eq:4}
     \left\{
       \begin{array}{ll}
         \dot p_i = v\left(\frac{\ell}{p_{i+1}- p_i}\right)
         \qquad i=1, \ldots, n-1
         \\
         \dot p_{n} = w (t)
         \\
         p_i (0) = \bar p_i
       \end{array}
     \right.
   \end{equation}
   admits a unique solution $p=p(t)$ defined for all $t \in \reali^+$
   and attaining values in $\mathcal{P}_\ell^n$. Moreover, if $w' \in
   \Lloc1 (\reali^+; \reali^+)$, $\bar p' \in \mathcal{P}_\ell$ and
   $p' = p'(t)$ is the corresponding solution to~\eqref{eq:4}, the
   following stability estimate holds
   \begin{equation}
     \label{eq:estimate}
     \norma{p(t) - p'(t)}
     \leq
     \left(
       \norma{ \bar p - \bar p'}
       +
       \norma{w-w'}_{\L1 ([0,t])}
     \right)
     \exp{\left( 2\frac{\Lip(v)}{\ell}t\right) },
   \end{equation}
   for every $t\in \left] 0,+\infty\right[$.
 \end{lemma}

\begin{proof}
  By~\textbf{(v)}, the function $v$ can be extended to a bounded
  Lipschitz function $u$ defined on all $\reali$ setting
  \begin{equation}
    \label{eq:5}
    u (\rho)
    =
    \left\{
      \begin{array}{l@{\qquad\mbox{ if}\qquad}r@{\,}c@{\,}l}
        v (0) & \rho & < & 0
        \\
        v (\rho) & \rho & \in & [0, 1]
        \\
        0 & \rho & > & 1\,.
      \end{array}
    \right.
  \end{equation}
  Now we consider the Cauchy problem
  \begin{equation}
    \label{eq:4bis}
    \left\{
      \begin{array}{ll}
        \dot p_i = u\left(\frac{\ell}{p_{i+1}- p_i}\right)
        & i=1, \ldots, n-1
        \\
        \dot p_{n} = w (t)
        \\
        p_i (0) = \bar p_i
        & i=1, \ldots, n \,.
      \end{array}
    \right.
  \end{equation}
  By the standard ODE theory, there exists a $\C1$ solution $p = p
  (t)$ defined as long as $p_{i+1}-p_i >0$ for all $i=1, \ldots, n-1$.
  We now prove that in fact $p_{i+1}(t) - p_i(t) \geq l$ for every $t
  \geq 0$. To this aim, we assume by contradiction that there exists
  $t^{*}$ in $\reali^+$, such that $p_{i+1}(t^{*}) - p_i(t^{*})< l$.
  Then, since $p_{i+1}(0)-p_{i}(0)=\bar p_{i+1}- \bar p_{i}\geq l$,
  there exists $\bar t$ in $\reali^+$, with $\bar t<t^{*}$, such that
  $p_{i+1}(\bar t) - p_i(\bar t) = l$ and $p_{i+1}(t) -p_i(t)<l$ for
  every $t \in \left] \bar t, t^{*} \right]$. Since $u(\rho)=0$ for
  every $\rho>1$, for every $t \in \left] \bar t, t^{*} \right]$, we
  have
  \begin{displaymath}
    p_i(t)
    =
    p_i(\bar t) + \int_{\bar t}^{t} \! \dot p_i (s) \, ds
    =
    p_i(\bar t) + \int_{\bar t}^{t}
    u \left( \frac{\ell}{p_{i+1}(s) - p_i(s)} \right) ds
    =
    p_i(\bar t) .
  \end{displaymath}
  This yields a contradiction, since for every $t \in \left] \bar t,
    t^{*} \right]$ and for $i=1, \ldots,n-1$,
  \begin{displaymath}
    p_{i+1}(t) -p_i(t)
    \geq p_{i+1}(\bar t) - p_i( \bar t) = l,
  \end{displaymath}
  completing the existence proof.

  To prove the estimate~\eqref{eq:estimate}, observe that the right
  hand side in~\eqref{eq:4} is Lipschitz continuous, indeed
  \begin{equation}
    \label{eq:lip}
    \modulo{
      v\left(\frac{\ell}{p_{i+1}- p_i}\right)
      -
      v\left(\frac{\ell}{p_{i+1}'- p_i'}\right)
    }
    \leq
    \frac{\Lip(v)}{\ell}
    \left( \modulo{p_{i+1}-p_{i+1}'} + \modulo{p_{i}-p_{i}'}\right) \,.
  \end{equation}
  for $i = 1, \ldots, n-1$. Hence, by~\eqref{eq:lip},
  \begin{eqnarray}
    \nonumber
    \modulo{p_i (t) - p_i' (t)}
    & \leq &
    \modulo{\bar p_i - \bar p_i'}
    +
    \int_{0}^{t}
    \modulo{
      v\left(\frac{\ell}{p_{i+1}- p_{i}}\right)
      -
      v\left(\frac{\ell}{p_{i+1}'- p_{i}'}\right)
    } \d{s}
    \\
    \nonumber
    & \leq &
    \modulo{\bar p_i - \bar p_i'}
    +
    \frac{\Lip(v)}{\ell}
    \int_{0}^{t}
    \left( \modulo{p_{i+1}-p_{i+1}'} + \modulo{p_{i}-p_{i}'} \right) \d{s}
    \\
    \label{eq:i}
    & \leq &
    \norma{\bar p - \bar p'}
    +
    2\frac{\Lip(v)}{\ell}
    \int_{0}^{t} \norma{ p (s) -  p' (s)} \d{s} \,,
  \end{eqnarray}
  On the other hand, For $i=n$, we immediately have
  \begin{equation}
    \label{eq:n}
    \modulo{p_n (t) - p_n' (t)}
    \leq
    \norma{\bar p - \bar p'}
    +
    \norma{w - w'}_{\L1\left( \left[ 0,t\right] \right) } \,.
  \end{equation}
  Hence, \eqref{eq:i} and~\eqref{eq:n} together yield
  \begin{equation*}
    \norma{p (t) - p' (t)}
    \leq
    \norma{\bar p - \bar p'}
    +
    \norma{w - w'}_{\L1\left( \left[ 0,t\right] \right) }
    +
    2\frac{\Lip(v)}{\ell} \int_{0}^{t} \norma{ p (s) -  p' (s)} \d{s}
  \end{equation*}
  and an application of the usual Gronwall Lemma
  gives~\eqref{eq:estimate}.
\end{proof}

\begin{lemma}
  \label{lem:IBVP}
  Let $v$ satisfy~\textbf{(v)}. Fix $\gamma \in \C{0,1} (\reali^+
  \reali)$, $\bar \rho \in (\L1 \cap \BV) (\reali; [0,1])$ and $\tilde
  \rho \in (\L1 \cap \BV) (\reali^+; [0,1])$. Then, the initial --
  boundary value problem
  \begin{equation}
    \label{IBVP}
    \left\{
      \begin{array}{l@{\qquad}rcl}
        \partial_t \rho + \partial_x \left(\rho \, v (\rho)\right) =0
        & x  & < & \gamma (t)
        \\
        \rho (0,x) = \bar \rho (x)
        & x & < & \gamma (0)
        \\
        \rho\left(t, \gamma (t)\right) = \tilde \rho (t)
        & t & \geq  & 0 \,
      \end{array}
    \right.
  \end{equation}
  admits a unique weak entropy solution $\rho \in \C{0,1} (\reali^+;
  (\L1 \cap \BV) \left(\reali; [0,1]) \right)$.

  Moreover, there exists a constant $L > 0$ such that if $\gamma,
  \gamma' \in \C{0,1} (\reali^+ \reali)$ with $\Lip (\gamma), \Lip
  (\gamma') \leq V$ for a $V>0$, $\bar \rho, \bar \rho' \in (\L1 \cap
  \BV) (\reali; [0,1])$ and $\tilde \rho, \tilde \rho' \in (\L1 \cap
  \BV) (\reali^+; [0,1])$, then, the two solutions $\rho= \rho (t,x)$
  and $\rho' = \rho' (t,x)$ to~\eqref{IBVP} satisfy for all $t \in
  \reali^+$
  \begin{equation}
    \label{eq:estimate2}
    \norma{\rho(t) - \rho'(t)}_{\L1}
    \leq
    L\left(
      \norma{\bar \rho - \bar \rho'}_{\L1}
      +
      \norma{\gamma-\gamma' }_{\C0\left( \left[ 0,t\right] \right)}
      +
      (1+2V)
      \norma{\tilde \rho-\tilde \rho'}_{\L1\left( \left[ 0,t\right] \right)}
    \right).
  \end{equation}

\end{lemma}

The initial -- boundary value problem in~\eqref{IBVP} falls within the
framework of~\cite{ColomboGroli4}, see also~\cite{DuboisLefloch,
  LeFloch1}. Indeed, the scalar conservation law~\eqref{eq:LWR} is a
particular case of a Temple systems, see~\cite[(H1), (H2)
and~(H3)]{ColomboGroli4}. Hence, \cite[Theorem~2.3]{ColomboGroli4}
applies and Lemma~\ref{lem:IBVP} follows.

\begin{proofof}{Proposition~\ref{prop:LWR-FtL}}
  In~\eqref{eq:LWR-FtL}, the equations for $p_1, \ldots, p_n$ are
  decoupled from the partial differential equation for $\rho$. Hence,
  Lemma~\ref{lem:ode} applies and ensures the existence of $p = p
  (t)$, with $p (t) \in \mathcal{P}_\ell^n$, solving the ordinary
  differential system for all $t \in \reali^+$.  We then choose $\rho$
  as the solution to the initial -- boundary value problem
  \begin{equation}
    \label{eq:2}
    \left\{
      \begin{array}{l@{\qquad}rcl}
        \partial_t \rho + \partial_x \left(\rho \, v (\rho)\right) =0
        & (t,x) & \in &
        \left\{(t,x) \in \reali^+ \times \reali \colon x<p_1 (t)\right\}
        \\
        \rho (0,x) = \bar \rho (x)
        & x & < & \bar p_1
        \\
        \rho\left(t, p_1 (t)\right) = \frac{\ell}{p_2 (t) - p_1 (t)}
        & t & \in & \reali^+
      \end{array}
    \right.
  \end{equation}
  and we apply Lemma~\ref{lem:IBVP}, obtaining the existence of a map
  $\rho \in \C0 \left([0,T]; (\L1 \cap \BV) (\reali; [0,1] ) \right)$
  solving~\eqref{eq:2} in the usual sense
  of~\cite[Definition~2.1]{DuboisLefloch},
  \cite[Definition~C.1]{AmadoriColombo} or, equivalently,
  \cite[Definition~2.2]{ColomboGroli4}. Therefore,
  \ref{it:def:wes}.~and~\ref{it:def:bd} in Definition~\ref{def:Sol}
  hold, The
  requirements~\ref{def:it:ode}.~and~\ref{def:it:ode1}.~follow from
  Lemma~\ref{lem:ode}.

  The stability estimate related to the ordinary differential system
  follows from Lemma~\ref{lem:ode}. Concerning the partial
  differential equation, by~\eqref{eq:estimate2} we have
  \begin{eqnarray}
    \nonumber
    \norma{\rho(t,\cdot) - \rho'(t,\cdot)}_{\L1}
    & \leq &
    L\left( \norma{\bar \rho - \bar \rho'}_{\L1}+\norma{p_1 -p_1' }_{\C0\left( \left[ 0,t\right] \right) }\right)
    \\
    \label{eq:qui}
    & &
    +
    L ( 1 + 2 V)
    \norma{
      \frac{\ell}{p_2 (\cdot) - p_1 (\cdot)}
      -
      \frac{\ell}{p_2' (\cdot) - p_1' (\cdot)}}_{\L1\left([ 0,t] \right) }.
  \end{eqnarray}
  Compute the term in parentheses separately
  \begin{eqnarray*}
    \norma{
      \frac{\ell}{p_2 (\cdot) - p_1 (\cdot)}
      -\frac{\ell}{p_2' (\cdot)-p_1' (\cdot)}}_{\L1([ 0,t]) }
    & \leq &
    \frac{1}{\ell}
    \int_{0}^{t}
    \left( \modulo{p_2-p_2'}+\modulo{p_1-p_1'}\right) \d{s}
    \\
    & \leq &
    \frac{2}{\ell}\int_{0}^{t} {\norma {p(s)-p'(s)}ds}
    \\
    &\leq &
    \frac{2}{\ell} \, t \, \norma {p-p'}_{\C0\left( \left[ 0,t\right] \right) }
  \end{eqnarray*}
  and inserting the above result in~\eqref{eq:qui},
  using~\eqref{eq:estimate}, we obtain:
  \begin{eqnarray*}
    & &
    \norma{\rho(t,\cdot) - \rho'(t',\cdot)}_{\L1}
    \\
    & \leq &
    L \norma{\bar \rho - \bar \rho'}_{\L1}
    +
    L\norma{p_{1} -p_{1}' }_{\C0\left( \left[ 0,t\right] \right) }
    +
    L \left( 1+2V\right)\frac{2}{\ell}t \norma {p-p'}_{\C0([ 0,t])}
    \\
    & \leq &
    L \norma{\bar \rho - \bar \rho'}_{\L1}
    +
    L\left( 1 +(1+2V) \frac{2}{\ell} \, t\right) \norma{p -p' }_{\C0([0,t])}
    \\
    & \leq &
    L \norma{\bar \rho - \bar \rho'}_{\L1}
    +
    L\left( 1 + (1+2V) \frac{2}{\ell} \, t \right)
    \left( \norma{\bar p - \bar p'} + \norma{w-w'}_{\L1 ([0,t])} \right)
    \exp\left( 2\frac{\Lip(v)}{\ell}t\right)
  \end{eqnarray*}
  completing the proof.
\end{proofof}

\begin{proofof}{Proposition~\ref{prop:FtL-LWR}}
  To construct a solution to~\eqref{eq:FtL-LWR}, we first
  apply~\cite[Theorem~6.3]{BressanLectureNotes} to obtain a Kru\v zkov
  solution $\rho = \rho (t,x)$ to the Cauchy problem for the scalar
  conservation law
  \begin{equation}
    \label{CP}
    \left\{
      \begin{array}{l@{\qquad}rcl}
        \partial_t \rho + \partial_x \left(\rho \, v (\rho)\right) =0
        & (t,x)  & \in & \reali^{+}\times \reali
        \\
        \rho(0,x)=
        \left\{
          \begin{array}{l@{\qquad\mbox{ if}\qquad}r@{\,}c@{\,}l}
            \bar \rho (x)& x & > & \bar p_{n}
            \\
            0 & x & < & \bar p_{n}\,.
          \end{array}
        \right.
      \end{array}
    \right.
  \end{equation}
  Then, we find the free boundary $p_n = p_n (t)$ solving the Cauchy
  problem for the ordinary differential equation
  \begin{equation}
    \label{CP1}
    \left\{
      \begin{array}{l}
        \dot p_n = v\left(\rho\left( t,p_{n}(t)\right) \right)
        \\
        p_{n}(0) = \bar p_{n}\,.
      \end{array}
    \right.
  \end{equation}
  The well posedness of~\eqref{CP1} is ensured
  by~\cite[Theorem~2.4]{ColomboMarson}, which we can apply due
  to~\textbf{(v)}, see also~\cite[Item~1 in Section~2]{ColomboMarson}.

  Next we restrict the solution $\rho = \rho(t,x)$ to~\eqref{CP} to
  $\left\lbrace (t,x)\in\reali^{+}\times \reali:
    x>p_{n}(t)\right\rbrace $. Then, we solve the following system of
  $n-1$ ordinary differential equations
  \begin{equation}
    \label{eq:6}
    \left\{
      \begin{array}{l}
        \dot p_i = v\left(\frac{\ell}{p_{i+1}- p_i}\right)
        \qquad i=1, \ldots, n-1
        \\
        p_{i}(0)=\bar p_{i}\,.
      \end{array}
    \right.
  \end{equation}

  By construction, \ref{it:def:wes}.~in
  Definition~\ref{def:Sol}~holds.  Condition~\ref{it:def:bd}.~is in
  this case empty. The requirement~\ref{def:it:ode1}.~is satisfied
  since $p_n$ solves~\eqref{CP1} and the previous application of
  Lemma~\ref{lem:ode} to~\eqref{eq:6} ensures~\ref{def:it:ode}.

  Passing to the stability estimates, using~\cite[(ii) in
  Theorem~6.3]{BressanLectureNotes}, we have
  \begin{eqnarray*}
    \norma{\rho (t) - \rho' (t)}_{\L1}
    & \leq &
    \int_{\reali} \modulo{
      \bar\rho (x) \, \caratt{[\bar p_n', +\infty[} (x)
      -
      \bar\rho' (x) \, \caratt{[\bar p_n', +\infty[} (x)
    }
    \d x
    \\
    & \leq &
    \norma{\bar \rho - \bar \rho'}_{\L1}
    +
    \modulo{\bar p_n - \bar p_n'}
    \\
    & \leq &
    \norma{\bar \rho - \bar \rho'}_{\L1}
    +
    \norma{\bar p - \bar p'} \,,
  \end{eqnarray*}
  proving~\eqref{eq:prima}. to prove~\eqref{eq:seconda}, we
  use~\cite[Theorem~2.2]{ColomboMarson} to obtain, in the case
  $\bar\rho = \bar \rho'$,
  \begin{displaymath}
    \modulo{p_n (t) - p_n' (t)}
    \leq
    c (t) \, \modulo{\bar p - \bar p'}^\alpha
  \end{displaymath}
  where $c$ is the constant exhibited in~\cite[Item~(2),
  Theorem~2.2]{ColomboMarson} with respect to the interval $[0,t]$ and
  , by~\cite[formula~(2.1)]{ColomboMarson},
  \begin{displaymath}
    1-\alpha
    \geq
    \max_{\rho \in [0,R]}
    \frac{v (\rho) - v (0)}{v (0) - v (\rho) - \rho\, v' (\rho)}
    \quad \mbox{or, equivalently} \quad
    \alpha
    =
    \left(1 + \max _{[0,R]}\frac{v (\rho) - v (0)}{\rho \, v' (\rho)}\right)^{-1}
  \end{displaymath}
  which is finite by~\textbf{(v)}. Finally, \eqref{eq:seconda}
  directly follows from Lemma~\ref{lem:ode}.
\end{proofof}

\begin{proofof}{Proposition~\ref{prop:F}}
  Use the integral form of the conservation law~\eqref{eq:LWR} in the
  region
  \begin{displaymath}
    \Omega
    =
    \left\{
      (\tau,\xi) \in \reali^+ \times \reali
      \colon
      \tau \in [0, t] \mbox{ and } \xi \in [p_{n_j}^j (t), p_1^{j+1} (t)]
    \right\}
  \end{displaymath}
  and obtain:
  \begin{eqnarray*}
    \int_{p_{n_j}^j (t)}^{p_1^{n_{j+1}} (t)} \rho (t, x) \d{x}
    \!\!\!& - &\!\!\!
    \int_{\bar p_{n_j}^j}^{\bar p_1^{j+1}} \bar \rho (x) \d{x}
    =
    \\
    & = &
    \int_0^t
    \left[
      \rho\left(\tau, p_{n_j}^j (\tau) \right)
      \quad
      (\rho \, v)\left(\tau, p_{n_j}^j (\tau) \right)
    \right]
    \left[
      \begin{array}{c}
        -\dot p_{n_j}^j (\tau) \\ 1
      \end{array}
    \right]
    \d\tau
    \\
    & &
    +
    \int_0^t
    \left[
      \rho\left(\tau, p_1^{j+1} (\tau) \right)
      \quad
      (\rho \, v)\left(\tau, p_1^{j+1} (\tau) \right)
    \right]
    \left[
      \begin{array}{c}
        \dot p_1^{j+1} (\tau) \\ -1
      \end{array}
    \right]
    \d\tau
    \\
    & = &
    \int_0^t
    \left[
      \rho\left(\tau, p_{n_j}^j (\tau) \right)
      \quad
      (\rho \, v)\left(\tau, p_{n_j}^j (\tau) \right)
    \right]
    \left[
      \begin{array}{c}
        -v\left(\tau, p_{n_j}^j (\tau) \right)\\ 1
      \end{array}
    \right]
    \d\tau
    \\
    & &
    +
    \int_0^t
    \frac{\ell}{p_2^{j+1} (\tau) - p_1^{j+1} (\tau)}
    \left[
      1
      \quad
      v \left(\frac{\ell}{p_2^{j+1} (\tau) - p_1^{j+1} (\tau)} \right)
    \right]
    \left[
      \begin{array}{c}
        \dot p_1^{j+1} (\tau) \\ -1
      \end{array}
    \right]
    \d\tau
    \\
    & = &
    0
  \end{eqnarray*}
  since $(p,\rho)$ solves~\eqref{eq:3} in the sense of
  Definition~\ref{def:Sol}.
\end{proofof}

{\small{

    \bibliographystyle{abbrv}

    \bibliography{microMACRO}

  }}

\end{document}